\documentclass[11pt]{amsart}
\usepackage{amssymb}
\usepackage{pgfplots}
\usepackage{csvsimple}
\usepackage{siunitx}
\usepackage{hyperref}
\usepackage{graphicx,color}
\usepackage[normalem]{ulem}
\usepackage{mathrsfs}
\usepackage[maxbibnames=99,isbn=false,backend=biber,
            url=false,eprint=true,style=alphabetic,
            giveninits=true,labelnumber,defernumbers]{biblatex}

\graphicspath{{figs/}}
\def\R{\mathbb{R}}

\def\Rinf{\R\cup \{+\infty\}}

\def\cI{\mathcal{I}}

\def\cL{\mathcal{L}}

\def\cN{\mathcal{N}}

\def\cS{\mathcal{S}}
\def\cT{\mathcal{T}}

\def\a{\alpha}

\def\d{\delta}

\def\p{\partial}

\def\veps{\varepsilon}

\def\vphi{\varphi}
\def\O{\Omega}

\def\G{\Gamma}
\def\GD{{\Gamma_D}}
\def\GN{{\Gamma_N}}

\def\ty{\widetilde{y}}

\def\ou{\overline{u}}

\newcommand{\dv}[1]{\,{\mathrm d}#1}

\newcommand{\wcheck}[1]{#1\hspace{-.8ex}\mbox{\huge {\lower.45ex \hbox{$\textstyle \check{}$}}} \hspace{.5ex}}

\DeclareMathOperator{\diver}{div}

\DeclareMathOperator{\id}{id}
\let\oldmarginpar\marginpar
\renewcommand\marginpar[1]{
  \oldmarginpar[\raggedleft\footnotesize #1]
  {\raggedright\footnotesize #1}}
\newtheorem{definition}{Definition}

\newtheorem{proposition}[definition]{Proposition}
\newtheorem{theorem}[definition]{Theorem}
\newtheorem{corollary}[definition]{Corollary}
\newtheorem{remark}[definition]{Remark}

\newtheorem{examples}[definition]{Examples}

\numberwithin{definition}{section}

\def\RT{{\mathcal{R}T}}
\def\CR{{cr}}

\def\ou{\overline{u}}

\def\nablah{\nabla_{\! h}}

\def\SRT{\RT^0_{\!N}(\cT_h)}
\def\SRTF{\RT^0(\cT_h)}
\def\SCR{\cS^{1,\CR}_D(\cT_h)}
\def\SCRF{\cS^{1,\CR}(\cT_h)}
\def\SL0{\cL^0(\cT_h)}

\begin{document}
\title[Orthogonality of finite element spaces]{Orthogonality
relations of Crouzeix--Raviart and Raviart--Thomas finite element spaces}
\author[S. Bartels]{S\"oren Bartels}
\address{Abteilung f\"ur Angewandte Mathematik,  
Albert-Ludwigs-Universit\"at Freiburg, Hermann-Herder-Str.~10, 
79104 Freiburg i.~Br., Germany}
\email{bartels@mathematik.uni-freiburg.de}
\date{\today}
\author[Z. Wang]{Zhangxian Wang}
\address{Abteilung f\"ur Angewandte Mathematik,  
Albert-Ludwigs-Universit\"at Freiburg, Hermann-Herder-Str.~10, 
79104 Freiburg i.~Br., Germany}
\email{zhangxian.wang@mathematik.uni-freiburg.de}
\renewcommand{\subjclassname}{
\textup{2010} Mathematics Subject Classification}
\begin{abstract}
Identities that relate projections of Raviart--Thomas finite element
vector fields to discrete gradients of Crouzeix--Raviart finite element 
functions are derived under general conditions. Various implications
such as discrete convex duality results and a characterization of
the image of the projection of the Crouzeix--Ravaiart space
onto elementwise constant functions are deduced.
\end{abstract}
  

\maketitle

\section{Introduction}\label{sec:intro}
Recent developments in the numerical analysis of total variation
regularized and related nonsmooth minimization problems show
that nonconforming and discontinuous finite element methods lead
to optimal convergence rates under suitable regularity 
conditions~\cite{ChaPoc19-pre,Bart20a-pre,Bart20b-pre}. 
This is in contrast to standard conforming methods
which often perform suboptimally~\cite{BaNoSa15}. A key 
ingredient in the derivation of quasi-optimal error estimates are 
discrete convex duality
results which exploit relations between Crouzeix--Raviart and
Raviart--Thomas finite element spaces introduced in~\cite{CroRav73}
and~\cite{RavTho77}. In particular, assume 
that $\O\subset \R^d$ is a bounded Lipschitz domain with a 
partitioning of the boundary into subsets $\GN,\GD\subset \p\O$, 
and let $\cT_h$ be a regular triangulation of $\O$. 
For a function $v_h\in \SCR$ and a vector field $y_h\in \SRT$
we then have the integration-by-parts formula 
\[
\int_\O \nablah v_h \cdot y_h \dv{x} = 
- \int_\O v_h \diver y_h \dv{x}.
\]
Important aspects here are that despite the possible discontinuity
of $v_h$ and $y_h$ no terms occur that are related
to interelement sides and that the vector
field $y_h$ and the function $v_h$ can be replaced by their 
elementwise averages on the
left- and right-hand side, respectively. In combination with 
Fenchel's inequality this implies a weak discrete duality relation. 

The validity of a strong discrete duality principle has been 
established in~\cite{ChaPoc19-pre,Bart20a-pre}
under certain differentiability or more generally approximability
properties of minimization problems using the orthogonality
relation
\begin{equation}\label{eq:rt_cr_ortho_a}
\big(\Pi_h \SRT\big)^\perp = \nablah \big(\ker \Pi_h |_{\SCR}\big),
\end{equation}
within the space of piecewise constant vector fields $\SL0^d$ 
equipped with the $L^2$ inner product and with $\nablah$ and
$\Pi_h$ denoting the elementwise application of the gradient
and orthogonal projection onto $\SL0^d$, respectively, $\ker$ denotes
the kernel of an operator.  
The identity implies that if a vector field $w_h\in \SL0^d$
satisfies
\[
\int_\O w_h \cdot \nablah v_h \dv{x} = 0
\]
for all $v_h \in \SCR$ with $\Pi_h v_h=0$ then there
exists a vector field $z_h \in \SRT$ such that
\[
w_h = \Pi_h z_h.
\]
Note that this is a stronger implication then the well known
result that if $w_h$ is orthogonal to discrete gradients
of {\em all} Crouzeix--Raviart functions then it belongs to
the Raviart--Thomas finite element space. 
Although strong duality is not required in the error analysis,
it reveals a compatibility property of discretizations and 
indicates optimality of estimates. Moreover, it is related
to postprocessing procedures that provide the solution 
of computationally expensive discretized dual problems 
via simple postprocessing procedures of numerical 
solutions of less expensive primal problems, 
cf.~\cite{Mari85,ArnBre85,CarLiu15,Bart20a-pre}. 

The proof of~\eqref{eq:rt_cr_ortho_a} given in~\cite{ChaPoc19-pre}
makes use of a discrete Poincar\'e lemma which is valid if
the Dirichlet boundary $\GD\subset \p\O$ is empty or
if $d=2$ and $\GD$ is connected. 
In this note we show that~\eqref{eq:rt_cr_ortho_a} can be 
established for general boundary partitions
by avoiding the use of the discrete Poincar\'e lemma. The
new proof is based on the surjectivity property of
the discrete divergence operator
\[
\diver : \SRT \to \SL0.
\]
This is a fundamental property for the use of the 
Raviart--Thomas method for discretizing saddle-point problems,
cf.~\cite{RavTho77,BoBrFo13-book}.
It is an elementary consequence of a projection property of
a quasi-interpolation operator $\cI_\RT:H^s(\O;\R^d)\to \SRT$
and the surjectivity of the divergence operator onto
the space $L^2(\O)$. 

Our arguments also provide a dual version of the orthogonality
relation~\eqref{eq:rt_cr_ortho_a} which states that
\begin{equation}\label{eq:rt_cr_ortho_b}
\diver \big(\ker \Pi_h |_{\SRT}\big) = \big(\Pi_h \SCR \big)^\perp. 
\end{equation}
Unless $\GD= \p\O$ we have that the left-hand side is trivial
and hence the identity yields that
\[
\Pi_h \SCR = \SL0,
\]
i.e., that the projection of Crouzeix--Raviart functions onto
elementwise constant functions is a surjection. If $\GD=\p\O$
then depending on the triangulation both equality or 
strict inclusion occur. This observation
reveals that the discretizations of total-variation regularized 
problems devised in~\cite{ChaPoc19-pre,Bart20a-pre} can be 
seen as discretizations
using elementwise constant functions with suitable nonconforming
discretizations of the total variation functional. 

The most important consequence of~\eqref{eq:rt_cr_ortho_b} is the
strong duality relation for the discrete primal problem defined
by minimizing the functional
\[
I_h(u_h) = \int_\O \phi(\nablah u_h) + \psi_h(x,\Pi_h u_h) \dv{x}
\]
in the space $\SCR$ and the discrete dual problem consisting in 
maximizing the functional
\[
D_h(z_h) = - \int_\O \phi^*(\Pi_h z_h) + \psi_h^*(x,\diver z_h) \dv{x}
\]
in the space $\SRT$. The functions $\phi$ and $\psi_h$ are suitable
convex functions with convex conjugates $\phi^*$  and $\psi_h^*$,
we refer the reader to~\cite{Bart20a-pre} for details. 

This article is organized as follows. In Section~\ref{sec:prelim} we
define the required finite element spaces along with certain
projection operators. Our main results are contained in 
Section~\ref{sec:ortho_rels}, where we prove the 
identities~\eqref{eq:rt_cr_ortho_a} and~\eqref{eq:rt_cr_ortho_b}
and deduce various corollaries. In the Appendix~\ref{sec:app_poincare}
we provide a proof of the discrete Poincar\'e lemma that leads to
an alternative proof of the main result under certain
restrictions.  

\section{Preliminaries}\label{sec:prelim}

\subsection{Triangulations}
Throughout what follows we let $(\cT_h)_{h>0}$ be a sequence of
regular triangulations of the bounded polyhedral Lipschitz domain 
$\O\subset \R^d$, cf.~\cite{BreSco08-book,Ciar78-book}. We let $P_k(T)$ denote
the set of polynomials of maximal total degree $k$ on $T\in \cT_h$ and
define the set of elementwise polynomial functions or 
vector fields
\[
\cL^k(\cT_h)^\ell = \{ w_h \in L^\infty(\O;\R^\ell): w_h|_T \in P_k(T) 
\text{ for all }T\in \cT_h\}.
\]
The parameter $h>0$ refers to the maximal mesh-size of the triangulation
$\cT_h$. The set of sides of elements is denoted by $\cS_h$. We let
$x_S$ and $x_T$ denote the midpoints (barycenters) of sides and elements, 
respectively. The $L^2$ projection onto piecewise constant functions or 
vector fields is denoted by
\[
\Pi_h : L^1(\O;\R^\ell) \to \SL0^\ell.
\]
For $v_h\in \cL^1(\cT_h)^\ell$ we have $\Pi_h v_h|_T = v_h(x_T)$ for all 
$T\in \cT_h$. We repeatedly use that $\Pi_h$ is self-adjoint, i.e., 
\[
(\Pi_h v,w) = (v,\Pi_h w)
\]
for all $v,w\in L^1(\O;\R^\ell)$ with the $L^2$ inner product $(\cdot,\cdot)$.

\subsection{Crouzeix--Raviart finite elements}
The Crouzeix--Raviart finite element space introduced in~\cite{CroRav73} 
consists of piecewise affine functions that are continuous at the midpoints of
sides of elements, i.e., 
\[
\SCRF = \{v_h \in \cL^1(\cT_h): v_h \text{ continuous in 
$x_S$ for all $S\in \cS_h$} \}.
\]
The elementwise application of the gradient operator
to a function $v_h\in \SCRF$ defines an elementwise
constant vector field $\nablah v_h$ via 
\[
\nablah v_h|_T = \nabla (v_h|_T)
\]
for all $T\in \cT_h$. For 
$v\in W^{1,1}(\O)$ we have $\nablah v = \nabla v$. 
Functions vanishing at midpoints
of boundary sides on $\GD$ are contained in 
\[
\SCR = \{v_h\in \SCRF: v_h(x_S)=0
\text{ for all $S\in \cS_h$ with $S\subset \GD$}\}.
\]
A basis of the 
space $\SCRF$ is given by the functions 
$\vphi_S \in \SCRF$, $S\in \cS_h$, satisfying the Kronecker property
\[
\vphi_S(x_{S'}) = \d_{S,S'}
\]
for all $S,S'\in \cS_h$. The function $\vphi_S$ vanishes on elements that
do not contain the side $S$ and is continuous with value~1 along $S$. A
quasi-interpolation operator is for $v\in W^{1,1}(\O)$ defined via
\[
\cI_\CR v = \sum_{S\in \cS_h} v_S \vphi_S, \quad v_S = |S|^{-1} \int_S v \dv{s},
\]
We have that $\cI_\CR$ preserves averages of gradients, i.e., 
\[
\nablah \cI_\CR v = \Pi_h \nabla v,
\]
which follows from an integration by parts, cf.~\cite{BoBrFo13-book,Bart16-book}.

 
\subsection{Raviart--Thomas finite elements}
The Raviart--Thomas finite element space of~\cite{RavTho77} 
is defined as
\[\begin{split}
\SRTF  = \{y_h\in & H(\diver;\O): y_h|_T(x) = a_T + b_T (x-x_T), \\
&  a_T\in \R^d, \, b_T\in \R \text{ for all $T\in\cT_h$} \},
\end{split}\]
where $H(\diver;\O)$ is the set of $L^2$ vector fields whose distributional
divergence belongs to $L^2(\O)$.
Vector fields in $\SRTF$ have continuous constant normal components
on element sides. The subset of vector fields with
vanishing normal component on the Neumann boundary $\GN$ is defined as
\[
\SRT = \{ y_h\in \SRTF: y_h \cdot n = 0 \text{ on $\GN$}\},
\]
where $n$ is the outer unit normal on $\p\O$. A basis of the
space $\SRTF$ is given by vector fields $\psi_S$ associated with
sides $S\in \cS_h$. Each vector field $\psi_S$ is 
supported on adjacent elements $T_\pm \in \cT_h$ with
\begin{equation}\label{eq:def_rt_basis}
\psi_S(x) = \pm \frac{|S|}{d! |T_\pm|} (z_{S,T_\pm} - x)
\end{equation}
for $x\in T_\pm$ with opposite vertex $z_{S,T_\pm}$ to $S\subset \p T_\pm$.
We have the Kronecker property 
\[
\psi_S|_{S'} \cdot n_{S'} = \d_{S,S'}
\] 
for all sides $S'$ with unit normal vector $n_{S'}$, if $S'=S$ we assume
that $n_S$ points from $T_-$ into $T_+$. A quasi-interpolation operator is 
for vector fields $z\in W^{1,1}(\O;\R^d)$ given by 
\[
\cI_\RT z = \sum_{S\in\cS_h} z_S \psi_S, \quad 
z_S = |S|^{-1} \int_S z \cdot n_S \dv{s}.
\]
For the operator $\cI_\RT$ we have the projection property
\[
\diver \cI_\RT z = \Pi_h \diver z,
\]
which is a consequence
of an integration by parts, cf.~\cite{BoBrFo13-book,Bart16-book}. 
This identity implies that the divergence 
operator defines a surjection
from $\SRT$ into $\SL0$, provided that constants
are eliminated from $\SL0$ if $\GD = \emptyset$. 

\subsection{Integration by parts}
An elementwise integration by parts implies that for $v_h\in \SCRF$
and $y_h\in \SRTF$ we have the integration-by-parts formula
\begin{equation}\label{eq:int_parts_rt_cr}
\int_\O \nabla_h v_h \cdot y_h \dv{x} + \int_\O v_h \diver y_h \dv{x} 
= \int_{\p\O}  v_h \, y_h \cdot n \dv{s}.
\end{equation}
Here we used that $y_h$ has continuous constant normal components on 
inner element sides and that jumps of $v_h$ have vanishing integral mean. If
an elementwise constant vector field $w_h\in \SL0^d$ satisfies
\[
\int_\O w_h \cdot \nablah v_h \dv{x} = 0
\]
for all $v_h\in \SCR$ then by choosing $v_h = \vphi_S$ for $S\in \cS_h\setminus \GD$
one finds that its normal components are 
continuous on inner element sides and vanish on the $\GN$, so that 
$w_h \in \SRT$.  We thus have the decomposition 
\[
\SL0^d = \ker (\diver|_{\SRT})  \oplus \nablah \SCR,
\]
where we used that $\ker (\diver|_{\SRT}) = \SL0^d\cap \SRT$.

\section{Orthogonality relations}\label{sec:ortho_rels}

The following identities and in particular their proofs and 
corollaries are the main contributions of this article. 

\begin{theorem}[Orthogonality relations]\label{thm:ortho_rels}
Within the sets of elementwise constant vector fields and functions 
$\cL^0(\cT_h)^\ell$ equipped with the $L^2$ inner product we have
\[\begin{split}
\big(\Pi_h \SRT\big)^\perp &= \nablah \big(\ker \Pi_h|_{\SCR} \big), \\
\diver \big(\ker \Pi_h |_{\SRT} \big) &= \big(\Pi_h \SCR \big)^\perp.
\end{split}\]
\end{theorem}

\begin{proof}
(i) The integration-by-parts formula~\eqref{eq:int_parts_rt_cr} implies
\[
(\nablah v_h,\Pi_h y_h) = - (v_h, \diver y_h) = - (\Pi_h v_h,\diver y_h) = 0
\]
if $\Pi_h v_h =0$ and hence 
$\Pi_h \SRT \subset \big[\nablah \big(\ker \Pi_h|_{\SCR} \big)\big]^\perp$.
To prove the converse inclusion let $y_h \in \SL0^d$ be orthogonal to 
$\nablah \big(\ker \Pi_h|_{\SCR} \big)$. We show that there exists
$\ty_h\in \SRT$ with $\Pi_h \ty_h = y_h$. For this, let 
$Z_h = \big(\ker \Pi_h|_{\SCR} \big)^\perp \subset \SCR$ 
and $r_h \in Z_h$ be the uniquely defined function with
\[
(\Pi_h r_h,\Pi_h v_h) = (y_h,\nablah v_h)
\]
for all $v_h\in Z_h$. The identity holds for all $v_h\in \SCR$ since 
$y_h$ is orthogonal to discrete gradients of functions $v_h \in \SCR$ 
with $\Pi_h v_h = 0$. In particular, $\Pi_h r_h$ is orthogonal to constant
functions if $\GD = \emptyset$.  We choose $z_h \in \SRT$ with 
$-\diver z_h = \Pi_h r_h$ and verify that
\[
(y_h - z_h,\nablah v_h) = (\Pi_h r_h,\Pi_h v_h) + (\diver z_h,v_h)  = 0
\]
for all $v_h \in \SCR$. We next define 
$\ty_h|_T = y_h|_T + \diver z_h|_T (x-x_T)/d$ for all $T\in \cT_h$ and 
note that
\[
(\ty_h - z_h,\nablah v_h) = (y_h-z_h,\nablah v_h) = 0
\]
for all $v_h\in \SCR$. Since $\ty_h -z_h$ is elementwise constant, it
follows that $\ty_h-z_h\in \SRT$ and in particular $\ty_h \in \SRT$. 
By definition of $\ty_h$ we have $\Pi_h \ty_h = y_h$ which proves the
first asserted identity. \\
(ii) For the second statement we first note that if $\Pi_h y_h = 0$ 
for $y_h\in \SRT$ then 
\[
(\Pi_h v_h,\diver y_h) = (v_h, \diver y_h) = - (\nablah v_h,\Pi_h y_h) = 0
\]
for all $v_h\in \SCR$ and hence 
$\diver y_h \in \big(\Pi_h \cS^{1,\CR}_D(\cT_h) \big)^\perp$. It remains to show that
\[
\big(\Pi_h \cS^{1,\CR}_D(\cT_h) \big)^\perp 
\subset \diver \big(\ker \Pi_h |_{\RT^0_N(\cT_h)}\big).
\]
If $w_h \in \SL0$ is orthogonal to $\Pi_h \cS^{1,\CR}_D(\cT_h)$
we choose $z_h \in \SRT$ with $\diver z_h = w_h$ and note that
\[
(\Pi_h z_h ,\nablah v_h) = (z_h,\nablah v_h) = -(w_h, v_h) = -(w_h,\Pi_h v_h) = 0
\]
for all $v_h\in \SCR$. This implies that $\Pi_h z_h\in \SRT$ and hence also
$y_h = z_h -\Pi_h z_h\in \SRT$. Since $\Pi_h y_h = 0$ and 
$\diver y_h = w_h$ we deduce the second identity. 
\end{proof}

An implication is a surjectivity property of the mapping 
$\Pi_h:\SCR\to \SL0$ if $\GD\neq \p\O$. 

\begin{corollary}[Surjectivity]\label{cor:surject_proj_cr}
If $\GD \neq \p\O$ then we have 
\[
\Pi_h \SCR = \SL0.
\]
Otherwise, the subspace $\Pi_h \SCR \subset \SL0$ 
has codimension at most one. 
\end{corollary}

\def\tSCR{\cS^{1,cr}_{D'}(\cT_h)}

\begin{proof}
(i) From Theorem~\ref{thm:ortho_rels} we deduce that the asserted
identity holds if and only if $\diver \ker \Pi_h|_{\SRT} = \{0\}$. 
Since 
\[
\ker \Pi_h|_{\SRT} = 
\{y_h\in \SRT: y_h|_T = b_T (x-x_T) \text{ f.a. } T\in \cT_h\},
\]
the latter condition is equivalent
to $\ker \Pi_h|_{\SRT} = \{0\}$. Let $T\in \cT_h$ such that 
a side $S_0 \subset \p T$ belongs to $\GN$,
i.e., we have $y_h|_T(x) = \sum_{j=0}^d \a_j (x-z_{S_j})$, where $z_{S_j}$ 
is the vertex of $T$ opposite to the side $S_j\subset \p T$, and with $\a_0 = 0$. 
If $y_h(x_T) =0$ then it follows that $\a_j = 0$ for $j=1,\dots,d$ since the
vectors $x_T-z_{S_j}$ are linearly independent. Starting from this element
we may successively consider neighboring elements to
deduce that $y_h|_T=0$ for all $T\in \cT_h$. \\
(ii) If $\GD =\p\O$ we may argue as in~(i) by removing one side $S\in \cS_h\cap \GD$
from $\GD$, define $\GD'= \GD\setminus S$, and using the larger space $\tSCR$.
We then have $\Pi_h \SCR  \subset \Pi_h \tSCR = \SL0$. The difference
is trivial if and only if $\Pi_h \vphi_S$ belongs to $\Pi_h \tSCR$.
\end{proof}

The following examples show that both equality or strict inequality can
occur if $\GD=\p\O$. 

\begin{examples}
(i) Let $L\in \{1,2\}$, $\cT_h = \{T_1,\dots,T_L\}$, $\overline{\O}=T_1\cup \dots \cup T_L$,
$\GD = \p\O$. Then  $\Pi_h \SCR \simeq \R^{L-1}$ while $\SL0 \simeq \R^L$. \\
(ii) Let $\cT_h = \{T_1,T_2,T_3\}$ be a triangulation consisting of the
subtriangles obtained by connecting the vertices of a macro triangle $T$ 
with its midpoint $x_T$. Let $\overline{\O}=T_1\cup T_2 \cup T_3$ and 
$\GD = \p\O$. We then have $\Pi_h \SCR =\SL0$. 
\end{examples}

The second implication concerns discrete versions of convex duality relations. 
We let 
\[
\phi^*(s) = \sup_{r\in \R^\ell} s\cdot r - \phi(r)
\]
be the convex conjugate of a given convex function $\phi \in C(\R^d)$. 

\begin{corollary}[Convex conjugation]\label{cor:duality}
Let $\ou_h\in \Pi_h \SCR$ and $\phi\in C(\R^d)$ be convex. We then have
\[\begin{split}
\inf & \Big\{ \int_\O \phi(\nablah u_h) \dv{x}: u_h \in \SCR, \, \Pi_h u_h = \ou_h \Big\} \\
& \ge  \sup \Big\{ - \int_\O \phi^*(\Pi_h z_h) \dv{x} - (\ou_h,\diver z_h): z_h \in \SRT\Big\}.
\end{split}\]
If $\phi \in C^1(\R^d)$ and the infimum is finite then equality holds. 
\end{corollary}

\begin{proof}
An integration by parts and Fenchel's inequality show that
\begin{equation}\label{eq:fenchel_by_parts}
- (\Pi_h u_h,\diver z_h) =  (\nablah u_h, \Pi_h z_h) \le \phi(\nablah u_h) + \phi^*(\Pi_h z_h).
\end{equation}
This implies that the left-hand side is an upper bound for the right-hand
side. If $\phi$ is differentiable $u_h \in \SCR$ is optimal in the infimum
then we have the optimality condition
\[
\int_\O \phi'(\nablah u_h) \cdot \nablah v_h \dv{x} = 0
\]
for all $v_h\in \SCR$ with $\Pi_h v_h = 0$. Theorem~\ref{thm:ortho_rels}
yields that 
$\phi'(\nablah u_h) = \Pi_h z_h$ for some $z_h\in \SRT$. This identity implies 
equality in~\eqref{eq:fenchel_by_parts} and hence 
\[
\int_\O \phi(\nablah u_h) \dv{x} 
= -\int_\O \phi^*(\Pi_h z_h) \dv{x} - (\ou_h,\diver z_h)
\]
so that the asserted equality follows.
\end{proof}

\begin{remark}
For nondifferentiable functions $\phi$, the strong duality relation 
can be established if there exists a sequence of continuously differentiable
functions $\phi_\veps$ such that the corresponding discrete primal 
and dual problems $I_{h,\veps}$ and $D_{h,\veps}$ are $\G$-convergent 
to $I_h$ and $D_h$ as $\veps\to 0$, respectively. An example is the 
approximation of $\phi(s)= |s|$ by functions 
$\phi_\veps(s) = \min\{|s|-\veps/2,|s|^2/(2\veps)\}$ for $\veps>0$.
\end{remark}
 
With the conjugation formula we obtain a canonical definition
of a discrete dual variational problem. 

\begin{corollary}[Discrete duality]\label{cor:discr_dual}
Assume that $\phi \in C(\R^d)$ is convex and
$\psi_h:\O\times \R \to \Rinf$ is elementwise constant in the first argument
and convex with respect to the second argument. For 
$u_h\in \SCR$ and $z_h \in \SRT$ define 
\[\begin{split}
I_h(u_h) &= \int_\O \phi(\nablah u_h) + \psi_h(x,\Pi_h u_h) \dv{x}, \\
D_h(z_h) &= - \int_\O \phi^*(\Pi_h z_h) + \psi_h^*(x,\diver z_h) \dv{x}.
\end{split}\]
We then have 
\[
\inf_{u_h\in \SCR} I_h(u_h) \ge \sup_{z_h\in \SRT} D_h(z_h).
\]
\end{corollary}

\begin{proof}
Using the result of Corollary~\ref{cor:duality} and exchanging
the order of the extrema we find that
\[\begin{split}
\inf_{u_h} I_h(u_h) 
&\ge \inf_{u_h} \sup_{z_h} -\int_\O \phi^*(\Pi_h z_h) \dv{x} 
 - (\Pi_h u_h,\diver z_h) + \int_\O \psi_h(x,\Pi_h u_h) \dv{x} \\
&\ge \sup_{z_h} -\int_\O \phi^*(\Pi_h z_h)\dv{x} 
+ \inf_{u_h} - (\Pi_h u_h,\diver z_h) + \int_\O \psi_h(x,\Pi_h u_h) \dv{x} \\
&= \sup_{z_h} -\int_\O \phi^*(\Pi_h z_h)\dv{x}
- \sup_{u_h} \, (\Pi_h u_h,\diver z_h) - \int_\O \psi_h(x,\Pi_h u_h) \dv{x} \\
&\ge \sup_{z_h} -\int_\O \phi^*(\Pi_h z_h)\dv{x}
 - \int_\O \psi_h^*(x,\Pi_h u_h) \dv{x} \\
&= \sup_{z_h} D_h(z_h).
\end{split}\]
This proves the asserted inequality. 
\end{proof}

The fourth implication concerns the postprocessing of solutions of the
primal problem to obtain a solution of the dual problem. This also
implies a strong discrete duality relation. 

\begin{corollary}[Strong discrete duality]
In addition to the conditions of Corollary~\ref{cor:discr_dual} 
assume that $\phi \in C^1(\R^d)$ and 
$\psi_h :\O\times \R^d \to \R$ is finite and differentiable with respect to 
the second argument. 
If $u_h\in \SCR$ is minimal for $I_h$ then the vector field
\[
z_h = \phi'(\nablah u_h) + \psi_h'(x,\Pi_h u_h) d^{-1} (1-\Pi_h) \id
\]
is maximal for $D_h$ with $I_h(u_h) = D_h(z_h)$. 
\end{corollary}

\begin{proof}
The optimal $u_h\in \SCR$ solves the optimality condition
\begin{equation}\label{eq:opt_min}
\int_\O \phi'(\nablah u_h) \cdot \nablah v_h + \psi_h'(\cdot,\Pi_h u_h) \Pi_h v_h \dv{x}
= 0
\end{equation}
for all $v_h \in \SCR$. By restricting to functions satisfying
$\Pi_h v_h =0$ we deduce with Theorem~\ref{thm:ortho_rels}
that there exists $z_h \in \SRT$ with 
\[
\Pi_h z_h = \phi'(\nablah u_h).
\]
The optimality condition~\eqref{eq:opt_min} implies that 
$\diver z_h = \psi_h' (\cdot,\Pi_h u_h)$. Hence, $z_h$ satisfies 
the asserted identity. With the resulting Fenchel identities 
\[\begin{split}
\nablah u_h \cdot \Pi_h z_h &= \phi(\nablah u_h) + \phi^*(\Pi_h z_h),\\
\Pi_h u_h \cdot \diver z_h &= \psi_h(\cdot, \Pi_h u_h) + \psi_h^*(\cdot,\diver z_h),
\end{split}\]
and by choosing $v_h = u_h$ in~\eqref{eq:opt_min} we find that
\[
I_h(u_h) = D_h(z_h)
\]
which in view of the weak duality relation 
$\inf_{u_h} I_h(u_h) \ge \sup_{z_h} D_h(z_h)$ implies that $z_h$ is optimal.
\end{proof}

\appendix
\section{Discrete Poincar\'e lemma}\label{sec:app_poincare}
For completeness we provide a derivation of~\eqref{eq:rt_cr_ortho_a} 
based on a discrete Poincar\'e lemma. We say that $\GD$ is connnected
if its relative interior has at most one connectivity component. 

\begin{proposition}[Discrete Poincar\'e lemma]\label{prop:cr_grads}
Assume that $\GD=\emptyset$ or $d=2$ and $\GD$ is connected. 
A vector field $w_h\in \SL0^d$ satisfies $w_h = \nablah v_h$
for a function $v_h \in \SCR$ if and only if  
\[ 
\int_\O w_h \cdot y_h \dv{x} = 0
\]
for all $y_h \in \SRT$ with $\diver y_h = 0$. 
\end{proposition}

\begin{proof}
If $w_h = \nablah v_h$ then the orthogonality relation follows from the 
inte\-gration-by-parts formula~\eqref{eq:int_parts_rt_cr}.
Conversely, if $w_h$ is orthogonal to 
vector fields $y_h \in \SRT$ with vanishing divergence then we can
construct a function $v_h$ by integrating $w_h$ along a path connecting
midpoints of sides, i.e., choosing a side $S_0$ at which some value is assigned
to $v_h$, e.g., $v_h(x_{S_0})=0$. If $\GD\neq \emptyset$ then we choose
$S_0\subset \GD$. The values at other sides are obtained via 
\[
v_h(x_S) = v_h(x_{S_0}) + \sum_{j=1}^J w_h|_{T_j} \cdot (x_{S_j}- x_{S_{j-1}}),
\]
where $(T_j)_{j=1,\dots,J}$ is a chain of (unique) elements connecting $S_0 \subset T_1$ 
with $S=S_J\subset T_J$ via the shared sides $S_1,\dots,S_{J-1}$, i.e., 
$T_j \cap T_{j+1} = S_j$ for $j=1,\dots,J-1$. To see that this is well defined it suffices
to show that for every closed path with $S_J=S_0$ the sum equals zero. To verify
this we define the Raviart--Thomas vector field 
\[
y_h = \sum_{j=1}^J \frac{(d-1)!}{|S_j|} \psi_{S_j},
\]
where we assume that the plus sign in 
\[
\psi_S(x) = \pm \frac{|S|}{d! |T_\pm|} (z_{S,T_\pm} - x)
\]
occurs for $\psi_{S_j}$ occurs on $T_{j+1}$
and the minus sign on $T_j$. For the element $T_j$ we then have that 
\[\begin{split}
\int_{T_j} w_h \cdot y_h \dv{x} 
&= w_h|_{T_j} \cdot (d-1)! \int_{T_j}  \big( |S_{j-1}|^{-1} \psi_{S_{j-1}} + |S_j|^{-1} \psi_{S_j}\big) \dv{x} \\
&= w_h|_{T_j} \cdot (d-1)! |T_j| \big(|S_{j-1}|^{-1} \psi_{S_{j-1}}(x_{T_j}) + |S_j|^{-1} \psi_{S_j}(x_{T_j})\big) \\
&= w_h|_{T_j} \cdot  d^{-1} \big((z_{S_{j-1}}) - x_{T_j}) - ( z_{S_j} - x_{T_j})\big) \\
&= w_h|_{T_j} \cdot (x_{S_j} - x_{S_{j-1}}),
\end{split}\]
where we used that $(z_{S,T} -x_T) = d (x_T-x_S)$ for $S\subset \p T$. Moreover, 
we have
\[
\diver y_h|_{T_j} 
= \diver \big((d-1)! |S_{j-1}|^{-1} \psi_{S_{j-1}} + (d-1)! |S_j|^{-1} \psi_{S_j} \big) =  0.
\]
This implies that $\diver y_h=0$ and hence by the assumed orthogonality
\[
\sum_{j=1}^J w_h|_{T_j} \cdot (x_{S_j}- x_{S_{j-1}}) 
= \sum_{j=1}^J \int_{T_j} w_h \cdot y_h \dv{x}  
= \int_\O w_h \cdot y_h \dv{x}
=0
\]
for every closed path of elements. Hence, the function $v_h$ is well defined
with $\nabla_h v_h = w_h$. 
If $d=2$ and $\GD$ is connected then by letting $\vphi_z \in C(\overline{\O})$ 
be an elementwise affine nodal basis function associated with an inner node
$z\in \cN_h\cap \GD$, i.e., $z=S_1\cap S_2$ for $S_1,S_2\in \cS_h \cap \GD$,
and choosing  $y_h = (\nabla \vphi_z)^\perp \in \SRT$, where 
$(a_1,a_2)^\perp = (-a_2,a_1)$, it follows that 
\[\begin{split}
0 &= \int_\O \nablah v_h \cdot (\nabla \vphi_z)^\perp \dv{x}
= \int_{S_1 \cup S_2} v_h (\nabla \vphi_z)^\perp \cdot n \dv{s} \\
&= \pm \big( v_h(x_{S_2}) - v_h(x_{S_1})\big),
\end{split}\]
i.e., that 
$v_h$ is constant on $\GD$. Here we used that $(\nabla \vphi_z)^\perp \cdot n$ is 
the tangential derivative on $\GD$ given by $\pm 1/|S_j|$ for $j=1,2$. 
\end{proof}

To deduce~\eqref{eq:rt_cr_ortho_a} from the proposition we argue as
in~\cite{ChaPoc19-pre} and 
let $w_h\in \SL0^d$ be orthogonal to $\Pi_h \SRT$ and hence also to
$\SRT$. Proposition~\ref{prop:cr_grads} implies that $w_h = \nabla_h v_h$ and
it remains to show that $v_h$ has the same value at all element midpoints. 
This follows from  
\[
0 = \int_\O \nablah v_h \cdot \psi_S \dv{x} = -\int_\O v_h \diver \psi_S \dv{x}
= \frac{|S|}{(d-1)!} \big(v_h(x_{T_+}) - v_h(x_{T_-})\big).
\]
Hence, $w_h\in \nablah \big(\ker \Pi_h|_{\SCR}\big)$. 
Note that a nontrivial $v_h$ only exists on triangulations that can be
partitioned by two colors, e.g., consisting of halved squares with the
same diagonal. 
Conversely, if $v_h\in \SCR$ with $\Pi_h v_h =0$ then 
the integration-by-parts formula~\eqref{eq:int_parts_rt_cr}
yields that $\nablah v_h$ is orthogonal to $\SRT$ and in
particular to $\Pi_h \SRT$. 

\subsection*{Acknowledgments.} The authors are grateful to Antonin Chambolle
for stimulating discussions and valuable hints. 

\section*{References}
\printbibliography[heading=none]

\end{document}